\documentclass[12pt,reqno]{amsart}
\UseRawInputEncoding

\usepackage{amsmath,amssymb,amscd,amsfonts}

\usepackage{dsfont}
\usepackage[all]{xy}
\usepackage{enumerate}

{\catcode`\@=11
\gdef\n@te#1#2{\leavevmode\vadjust{%
 {\setbox\z@\hbox to\z@{\strut#1}%
  \setbox\z@\hbox{\raise\dp\strutbox\box\z@}\ht\z@=\z@\dp\z@=\z@%
  #2\box\z@}}}
\gdef\leftnote#1{\n@te{\hss#1\quad}{}}
\gdef\rightnote#1{\n@te{\quad\kern-\leftskip#1\hss}{\moveright\hsize}}
\gdef\?{\FN@\qumark}
\gdef\qumark{\ifx\next"\DN@"##1"{\leftnote{\rm##1}}\else
 \DN@{\leftnote{\rm??}}\fi{\rm??}\next@}}

\DeclareFontFamily{OT1}{wncyr}{\hyphenchar\font45 }
\DeclareFontShape{OT1}{wncyr}{m}{n}{%
   <5> <6> <7> <8> <9> gen * wncyr
   <10> <10.95> <12> <14.4> <17.28> <20.74>  <24.88>wncyr10}{}
\DeclareFontShape{OT1}{wncyr}{m}{it}{%
   <5> <6> <7> <8> <9> gen * wncyi
   <10> <10.95> <12> <14.4> <17.28> <20.74> <24.88> wncyi10}{}
\DeclareFontShape{OT1}{wncyr}{m}{sc}{%
   <5> <6> <7> <8> <9> <10> <10.95> <12> <14.4>
   <17.28> <20.74> <24.88>wncysc10}{}
\DeclareFontShape{OT1}{wncyr}{b}{n}{%
   <5> <6> <7> <8> <9> gen * wncyb
   <10> <10.95> <12> <14.4> <17.28> <20.74> <24.88>wncyb10}{}
\input cyracc.def

\DeclareMathSizes{9}{9}{7}{5}

\DeclareMathSymbol{\twoheadrightarrow} {\mathrel}{AMSa}{"10}

\theoremstyle{plain}

\newtheorem{theorem}{Theorem}

\newtheorem{proposition}{Proposition}
\newtheorem{lemma}{Lemma}
\newtheorem{corollary}{Corollary}

\theoremstyle{definition}

\theoremstyle{remark}

\def\pp{{\mathbb P}}

\begin{document}

\title[Ideal of the variety of flexes of plane cubics]
{Ideal of the variety of flexes\\ of plane cubics}
\author[Vladimir L. Popov]{Vladimir L. Popov${}^{1}$}
\thanks{\today.
\newline \indent
${}^1$  Steklov Mathematical Institute,
Russian Academy of Sciences, Gub\-kina 8, Mos\-cow
119991, Russia, {\rm popovvl@mi-ras.ru}.
}

\begin{abstract}
We prove that the variety of flexes of algebraic curves of degree $3$ in the projective plane is
an ideal theoretic complete intersection in the product of a two-dimensional and a nine-dimen\-sional projective spaces.
\end{abstract}

\maketitle

\section{Introduction}\label{s1}

We fix an algebraically closed field $k$ of characteristic zero. Let $L$
be a $3$-dimensional vector space over $k$
and let $\Phi$ be the space ${\sf S}^3(L^*)$ of degree $3$ forms on $L$.  Let $x_0, x_1, x_2$ be a basis of $L^*$, and let
$\alpha_{i_0 i_1 i_2}$ be the elements of the basis of $\Phi^*$ dual to a basis of $\Phi$ consisting of monomials $x_0^{i_0}x_1^{i_1}x_2^{i_2}$.
The sets of forms $\{x_j\}$ and $\{\alpha_{i_0 i_1 i_2}\}$ are the projective coordinate systems
on the projective spaces $\mathbb{P}(L)$ and $\mathbb{P}(\Phi)$ associated with  $L$ and $\Phi$.

We denote by $\mathcal P$
the algebra of polynomials over $k$
in two groups of variables $\{x_j\}$ and $\{\alpha_{i_0 i_1 i_2}\}$.
The bi-homogeneous elements of $\mathcal P$ are the ones homogeneous in every of these groups; for such elements is well-defined their vanishing on $\mathbb{P}(L)\times \mathbb{P}(\Phi)$. Given a closed subset $S$ of $\mathbb{P}(L)\times \mathbb{P}(\Phi)$, the ideal
of $\mathcal P$ generated by all bi-homogeneous elements vanishing on $S$ is called {\it the ideal of $S$}.

We consider in $\mathcal P$ the following bi-homogeneous elements
\begin{align}
f&:=\textstyle \sum_{i_0+i_1+ i_2=3} \alpha_{i_0 i_1 i_2} x_0^{i_0}x_1^{i_1}x_2^{i_2},\label{ef}\\
h&:={\rm det}\Big(\frac{\partial^2f}{\partial x_i\partial x_j}\Big)\label{eh}.
\end{align}
whose bi-degrees are
\begin{equation}\label{dfh}
\deg(f)=(1, 3),\;\deg(h)=(3, 3).
\end{equation}
 These elements determine in $\mathbb{P}(L)\times \mathbb{P}(\Phi)$ the closed subset
\begin{equation}\label{F}
F:=\{a\in \mathbb{P}(L)\times \mathbb{P}(\Phi)\mid f(a)=h(a)=0\}
\end{equation}
called the {\it variety of flexes of plane cubics} because $F$ is a compactification of the set of all pairs $(e, E)$ where $E$ is a smooth plane cubic and $e$ is its flex; $F$ was explored in
\cite{Har}, \cite{Kul1}, \cite{Kul2}, \cite{Pop1}, \cite{Pop2}, \cite{Pop3}.

The main result of the present note is the following theorem.
\begin{theorem}\label{main}
 The ideal of $F$ is prime and generated by $f$ and $h$.
 \end{theorem}

  As $\dim F=9$ (see \cite{Pop2}, \cite{Pop3}), and $\dim \mathbb{P}(L)=2$, $\dim \mathbb{P}(\Phi)=9$, it follows from \eqref{F} that $F$ is a set theoretical complete intersection in $\mathbb{P}(L)\times \mathbb{P}(\Phi)$. Theorem \ref{main} yields that in fact a stronger property holds:
 \begin{corollary} The set
$F$
is an ideal theoretic complete intersection in $\mathbb{P}(L)\times \mathbb{P}(\Phi)$ of co\-dimen\-sion\;$2$.
\end{corollary}

The proof of Theorem \ref{main} is given in  Section \ref{s3}. Section \ref{s2} contains some auxiliary results about multi-cones in products of vector spaces.
The proof of Proposition\;\ref{p1} in this section
uses an idea from
a letter \cite{Ban} of Tatiana Bandman to me. I am sincerely grateful to her.

\vskip 2mm

{\it Conventions and notation}

\vskip 2mm

All algebraic varieties in this note are taken over the field $k$.

We use the standard notation and terminology from \cite{Bor}, \cite{Sha}. In particular, $k[X]$ denotes the $k$-algebra of all regular functions on an algebraic variety $X$.

$\pp(L)$ denotes the projective space of all one-dimensional linear sub\-spaces of a vector space $L$ over $k$, and
$\pi_L\colon L\setminus\{0\}\to \pp(L)$ denotes the canonical projection.

\section{Some properties of multi-cones}\label{s2}

We consider  the nonzero finite-dimensional vector spaces $V_1,\ldots, V_s$  over $k$, where $s\geqslant 2$, and put
\begin{equation}\label{V}
V:=V_1\times\cdots\times V_s.
\end{equation}
We also consider the $s$-dimensional torus
\begin{equation}\label{T}
T:=k^*\times \cdots\times k^*\;\;\mbox{($s$ factors)}.
\end{equation}
For every $v\in V$, $t\in T$, and $i=1,\ldots, s$, we denote by $v_i$ and $t_i$ respectively the projection of $v$ and $t$ to the $i$th factor of the right-hand side of \eqref{V} and \eqref{T}. We put
\begin{equation}\label{KU}
D:=\textstyle\bigcup_{i=1}^s
\{v\in V\mid v_i=0\},\quad U:=V\setminus D.
\end{equation}

The following formula determines the action $T\times V\to V$, $(t, v)\mapsto t\cdot v$ of $T$ on $V$:
\begin{equation}\label{acc}
(t\cdot v)_i:=t_iv_i\;\;\mbox{for all $i$.}
\end{equation}
All sets $\{v\in V\mid v_i=0\}$ and $U$ are stable with respect to this action and the $T$-stabilizer of every point of $U$ is trivial.

Let $q_1,\ldots, q_m$ be the multi-homogeneous regular functions on $V$. Let the multi-degree
of $q_i$ be $(d_{i,1},\ldots, d_{i,s})$, i.e.,
\begin{equation}\label{fi}
q_i(t\cdot v)=t_1^{d_{i,1}}\cdots t_s^{d_{i,s}}q_i(v)\;\;\mbox{for every $t\in T$, $v\in V$.}
\end{equation}
We assume that
\begin{equation}\label{r}
\mbox{$d_{i,j}>0$ for all $i, j$.}
\end{equation}

The functions
$q_1,\ldots, q_m$
determine the following closed subsets $Z$ and
$C$ of  respectively
$\pp(V_1)\times\cdots\times\pp(V_s)$
and $V$:
\begin{align}
Z&:=\{p\in \pp(V_1)\times\cdots\times\pp(V_s)\mid q_1(p)=\cdots=q_m(p)=0\},\label{Z} \\
C&:=\{v\in V\mid q_1(v)=\cdots=q_m(v)=0\}.
\label{C}
\end{align}
It follows from \eqref{fi} and
\eqref{C} that the set $C$ is a multi-cone, i.e.,
is
$T$-stable.

\begin{lemma}\label{ideals}
The ideal of $k[V]$ generated by all multi-homogeneous ele\-ments vanishing on $Z$ coincides with the ideal
\begin{equation}\label{ide}
\{g\in k[V]\mid
g(C)=0
\}.
\end{equation}
\end{lemma}
\begin{proof} For every character $\chi\colon T\to k^*$, the subspace
\begin{equation}\label{isoc}
k[V]_\chi:=\{g\in k[V]\mid g(t\cdot v)=\chi(t)g(v)\;\mbox{for all $t\in T$, $v\in V$}\}
\end{equation}
is an isotypic component of the $T$-module $k[V]$, and the latter is a direct sum of such components.
Since there are $l_{1},\ldots, l_{s}\in \mathbb Z$ such that $\chi(t)=t_1^{l_{1}}\cdots t_s^{l_{s}}$ for every $t\in T$, the elements of $V_{\chi}$ are precisely the multi-homogeneous regular functions on $V$ of the multi-degree $(l_1,\ldots, l_s)$.

Now let $g$ be a nonzero element of the ideal \eqref{ide}. Then
\begin{equation}\label{gi}
g=g_1+\cdots+g_r,
\end{equation}
where
$g_i\in  k[V]_{\chi_i}$, $g_i\neq 0$ for every $i$, and
\begin{equation}\label{disti}
\chi_i\neq\chi_j\;\;\mbox{for all $i\neq j$.}
\end{equation}

We claim that every $g_i$ vanishes on $C$, from which the assertion of the lemma obviously follows. To prove this, take a point $c\in C$.
From \eqref{ide}, \eqref{isoc}, \eqref{gi} and the $T$-invariance of $C$ it follows that
\begin{equation}\label{equaaa}
0=g(t\cdot c)=\chi_1(t)g_1(c)+\cdots+\chi_r(t)g_r(c)\;\;\mbox{for every $t\in T$.}
\end{equation}
In view of \eqref{disti} and Dedekind's theorem on linear in\-de\-pen\-dence of cha\-racters \cite[A.V.27]{Bou}, it follows from  \eqref{equaaa} that $g_i(c)=0$ for every $i$,
which completes the proof.
\end{proof}

The morphism $\pi_{V_1}\times \cdots\times \pi_{V_s}\colon
U\to \pp(V_1)\times\ldots\times\pp(V_s)$ is surjective. In view of  \eqref{Z}, \eqref{C},
its restriction to
$C\cap U$ is a surjective morphism
\begin{equation*}
\alpha\colon C\cap U\to Z.
\end{equation*}
The fibers of $\alpha$ are
$T$-orbits; whence
\begin{equation}\label{af}
\mbox{$\alpha^{-1}(\alpha(c))$  is irreducible and
$s$-dimensional
for every
$c\in C\cap U$.}
\end{equation}

By
\cite[p.\,74, Cor.\,1.14]{Sha},   for any irreducible components $Z'$ and $(C\cap U)'$ of respecti\-ve\-ly $Z$ and $C\cap U$,  in view of \eqref{Z}, \eqref{C} we have
\begin{align}
\dim Z'&\geqslant \dim V-s-m,\label{dZ}\\
\dim C' &\geqslant \dim V-m.\label{dC}
\end{align}

We denote
\begin{equation*}\label{clo}
\overline{C\cap U}:=\mbox{the closure
 of $C\cap U$ in
 $V$}.
\end{equation*}
 \begin{proposition}\label{p1}\

 \begin{enumerate}[\hskip 4mm\rm(a)]
\item $C=(C\cap U)\cup D$.
\item  $D\subseteq \overline{C\cap U}$ if $Z$ shares the following property:
\begin{equation}\label{*}
\left.
\begin{split}
&\mbox{the natural projection}\\
&\hskip 10mm
Z\to \pp(V_{1})\times\cdots\times \widehat{\pp(V_{i})}\times
\cdots\times \pp(V_{s})\\
& \mbox{is surjective for every $i=1,\ldots, s$.}
 \end{split}
 \right\}
 \end{equation}
  \end{enumerate}
 \end{proposition}
 \begin{proof} (a) It follows from \eqref{KU} that
 \begin{equation}\label{sub}
 C\subseteq (C\cap U)\cup D.
 \end{equation}
   By \eqref{acc}, \eqref{r}, every function $q_i$ vanishes on $D$; whence by \eqref{C} this yields
 \begin{equation}\label{sup}
  D\subseteq C.
 \end{equation}
 From \eqref{sub}, \eqref{sup} follows (a).

 (b) For every $i=1,\ldots, s$, the set
 \begin{equation}\label{k0}
\{v\in V\mid v_i=0, v_j\neq 0\;\mbox{for every $j\neq i$}\}
 \end{equation}
 is dense and open in the set $\{v\in V\mid v_i=0\}$. Whence, in view of \eqref{KU}, it suffices to prove that set \eqref{k0} is contained in $\overline{C\cap U}$.
  To prove this, take a point
  $a$ of this set.
  Since $Z$ shares property \eqref{*}, it follows from
 \eqref{k0} the existence of
  a point
 \begin{equation}\label{ccc}
 c\in C\cap U\;\;\mbox{such that $c_j=a_j$ for every $j\neq i$.}
 \end{equation}
 Consider in $T$ the one-dimensional subtorus
 \begin{equation}\label{Ti}
 T_i:=\{t\in T\mid t_j=1\;\mbox{for every $j\neq i$}\}.
 \end{equation}
  Since $c_i\neq 0$, it follows from  \eqref{acc} that
    the $T_i$-orbit of $c$
 is the curve
 \begin{equation*}\label{cur}
 \{u\in U\mid u_j=a_j\;\mbox{for every $j\neq i$}\}.
 \end{equation*}
In view of $T$-stability of the set $C\cap U$, this curve lies in $C\cap U$. Since the closure of this curve in
$V$ clearly contains $a$, this completes the proof.
 \end{proof}

 \begin{corollary} \label{coro1}
 If
  $Z$ shares property \eqref{*}, then
 \begin{enumerate}[\hskip 4mm\rm(a)]
 \item $C=\overline{C\cap U}$;
  \item $C$ is irreducible if and only if  $C\cap U$ is.
 \end{enumerate}
  \end{corollary}
\begin{proof} (a) Since $C$ and $D$ are closed in $V$, Proposition \ref{p1}(a) implies
that $C=(\overline{C\cap U})\cup D$; whence
(a) in view of Proposition \ref{p1}(b).

(b) This follows from (a).
\end{proof}

\begin{lemma}\label{p2} Let $\varphi\colon X\to Y$ be a morphism of algebraic varieties and let $d, n$ be positive integers such that the following hold:
\begin{enumerate}[\hskip 4mm\rm(a)]
\item Y is irreducible and $\dim Y=n$;
\item $\dim\varphi^{-1}(y)=d$ for every $y\in Y$;
\item $\varphi^{-1}(y)$ is irreducible for every $y$ from a dense open subset of $Y$;
\item the dimension of every irreducible component of $X$ is at least $n+d$.
\end{enumerate}
Then $X$ is irreducible.
\end{lemma}
\begin{proof}Let $X_1,\ldots, X_r$ be all irreducible components of $X$, let
$
X_i^0
:=\{x\in X_i\mid x\notin X_j\;\mbox{for every $j\neq i$}\},
$
and let $\varphi_i\colon X_i^0\to Y$ be the restriction of $\varphi$ to $X_i^0$. Then
\begin{equation}\label{empt}
X_i^0\cap X_j^0=\varnothing \;\; \mbox{if $j\neq i$},
\end{equation}
and, in view of (d), (b),
\begin{equation}\label{neqq}
\dim X_i^0\geqslant n+d,\;\; \dim \varphi_i^{-1}(\varphi_i(x))\leqslant d
\;\;\mbox{for every $x\in X_i^0$.}
\end{equation}
 By \cite[p.\,75, Thm.\,1.25]{Sha}, for every point $z$ of an open dense subset of $X_i^0$,
 we have
 \begin{align}
 \dim X_i^0-\dim \varphi_i^{-1}(\varphi_i(z))&=\dim \varphi_i(X_i^0),\label{tttt}\\
 \dim \varphi_i^{-1}(\varphi_i(z))&\leqslant \dim \varphi_i^{-1}(\varphi_i(x))\;\;\mbox{for all $x\in X_i^0$.}\label{ttttt}
 \end{align}
It follows from (a), \eqref{neqq}, \eqref{tttt} that  $\dim X_i^0=n+d$, and
\begin{gather}\label{ddddd}
\left.
\dim \varphi_i^{-1}(\varphi_i(z))=d,\;\;
\mbox{$\varphi_i$ is dominant}.
\right.
\end{gather}
From (b), \eqref{ttttt}, \eqref{ddddd} we deduce that
\begin{equation}\label{ddd}
\dim \varphi_i^{-1}(\varphi_i(x))=d \;\;\mbox{for every $x\in X_i^0$.}
\end{equation}
In turn, by (c) and \eqref{ddd},  for every point $y$ of a dense open subset of $Y$, the set $X_i^0\cap \varphi^{-1}(y)$ is open in the irreducible set $\varphi^{-1}(y)$. The irreducibility of $\varphi^{-1}(y)$ implies that different such sets cannot be disjoint.
This and \eqref{empt} yield $r=1$.
\end{proof}

\begin{corollary} \label{cccc}
We use the above notation. Let the following hold:
\begin{enumerate}[\hskip 4mm\rm(a)]
\item $Z$ is irreducible;
\item $Z$ shares property \eqref{*};
\item $\dim Z=\dim V-s-m$.
\end{enumerate}
Then $C$ is irreducible.
\end{corollary}
\begin{proof} In view of  (a), (c), and  \eqref{af}, the conditions of Lemma \ref{p2} are met for $X=C\cap U$, $Y=Z$, $\varphi =\alpha$, $d=s$, $n=\dim V-s-m$. Hence, by this lemma, $C\cap U$ is irreducible. In turn, by (b) and Corollary \ref{coro1}(b),  the latter implies irreducibility of\;$C$.
\end{proof}

\section{Proof of Theorem \ref{main}}\label{s3}

We use the notation from Sections \ref{s1}, \ref{s2} and set
\begin{equation}
\label{proo}
\left.
\begin{split}
s&=2, \,V_1=L, \,V_2=\Phi,\\
m&=2,\, q_1=f,\, q_2=h,\\
Z&=F.
\end{split}
\right\}
\end{equation}
In view of \eqref{dfh} condition \eqref{r} holds.

Since any two algebraic curves in the projective plane $\pp(L)$ intersect (see \cite[p.\,71, Cor.\,1.8]{Sha}), and each point of $\pp(L)$ clearly is a flex of some elliptic curve in $\pp(L)$,  both natural projections $\pp(L)\leftarrow F\to \pp(\Phi)$ are surjective (see also \cite{Pop2}). Hence $F$ shares property \eqref{*}. By \cite[Thms.\,1, 2]{Pop1}, \cite[Thms.\,1, 2]{Pop2}, the set $F$ is irreducible and $\dim F=9$ (the restriction $k=\mathbb C$ made in these papers is eliminated due to the Lefschetz principle). It then follows from \eqref{proo} and Corollary \ref{cccc} that the set
$C$ (see   \eqref{C}) is irreducible.

Now we will apply the following statement:

\begin{lemma}\label{ci}
Suppose
the polynomials
$g_1,\ldots, g_d\in k[\mathbb A^n]$ share the following properties:

\begin{enumerate}[\hskip 4.2mm\rm(a)]
\item
the set
$S:=\{a\in \mathbb A^n\mid g_1(a)=\ldots=g_d(a)=0\}$
 is ir\-re\-ducible;
\item there is
$a\in S$ such that the values $(dg_1)_a,\ldots,  (dg_d)_a$
at $a$ of the differential forms $dg_1,\ldots,  dg_d$ on $\mathbb A^n$ are linearly indepen\-dent over $k$.
\end{enumerate}
Then $S$ is an $(n-d)$-dimensional ideal theoretic complete intersection
in $\mathbb A^n$ and
$\{g\in k[\mathbb A^n]\mid g(S)=0
\}
$
is the ideal of $k[\mathbb A^n]$ generated by $g_1,\ldots, g_d$.
\end{lemma}

\begin{proof} See \cite[Lem.\,4]{Ko}.
\end{proof}

In view of \eqref{C}, \eqref{proo},  and the irreducibility of $C$, it follows from Lem\-mas\;\ref{ideals} and \ref{ci}
that to complete the proof of Theorem \ref{s3} it suffices to find a point $a\in C$ such that $(df)_a$, $(dh)_a$ are linearly in\-de\-pen\-dent over $k$. Below we find such a point.

For any polynomial $p\in \mathcal P
$, we
put
\begin{equation}\label{nota}
p_{(i_1\ldots i_s)}:=\frac{\partial^s p}{\partial x_{i_1}\ldots\partial x_{i_s}}.
\end{equation}
The polynomial $p_{(i_1\ldots i_s)}$ does not change under any permutation of indices $i_1,\ldots, i_s$. It is enough for us to find a point $a\in C$ such that
\begin{equation}\label{li}
{\rm det}\!
\begin{pmatrix}
f_{(0)}& f_{(1)}\\
h_{(0)}& h_{(1)}
\end{pmatrix}\!(a)\neq 0.
\end{equation}

Let $c\in V$ be the point determined by the conditions
\begin{equation}\label{point}
\begin{gathered}
x_0(c_1)=0,\;
x_1(c_1)=-1,\;
x_2(c_1)=1,\\
c_2=x_0^3+x_1^3+x_2^3+x_0x_1x_2.
\end{gathered}
\end{equation}
By \cite[Lem.\,6]{Pop2}, the point $c$ lies in $C$ (this can also be verified directly using
\eqref{C}, \eqref{ef}, \eqref{eh}). We shall show that  \eqref{li} holds for
\begin{equation}\label{ac}
a=c.
\end{equation}

 In view of \eqref{ef} we have
\begin{equation}\label{dF}
\left.
\begin{split}
f_{(0)} & = \alpha_{120}x_1^2 + \alpha_{102}x_2^2 + \alpha_{111}x_1x_2
\\&\hskip 
4.8mm
+ 2\alpha_{210}x_0x_1
 + 2\alpha_{201}x_0x_2 + 3\alpha_{300}x_0^2,\\
f_{(1)} & = \alpha_{210}x_0^2 + \alpha_{012}x_2^2 + \alpha_{111}x_0x_2
\\&\hskip 
4.8mm
+ 2\alpha_{120}x_0x_1 +
2\alpha_{021}x_1x_2 + 3\alpha_{030}x_1^2,\\
f_{(2)} & = \alpha_{201}x_0^2 + \alpha_{021}x_1^2 + \alpha_{111}x_0x_1
\\&\hskip 
4.8mm
+
2\alpha_{102}x_0x_2 + 2\alpha_{012}x_1x_2 + 3\alpha_{003}x_2^2,
\end{split}\;\right\}
\end{equation}

\begin{equation}\label{dou}
\left.\begin{split}
f_{(00)}&=2\alpha_{210}x_1+2\alpha_{201}x_2+6\alpha_{300}x_0,\\
f_{(01)}&=2\alpha_{120}x_1+\alpha_{111}x_2+2\alpha_{210}x_0,\\
f_{(02)}&=2\alpha_{102}x_2+\alpha_{111}x_1+2\alpha_{201}x_0,\\
f_{(11)}&=2\alpha_{120}x_0+2\alpha_{021}x_2+6\alpha_{030}x_1,\\
f_{(12)}&=2\alpha_{012}x_2+\alpha_{111}x_0+2\alpha_{021}x_1,\\
f_{(22)}&=2\alpha_{102}x_0+2\alpha_{012}x_1+6\alpha_{003}x_2.
\end{split}\;\right\}
\end{equation}

\begin{equation}\label{thr}
\left.
\begin{split}
f_{(000)}&=6\alpha_{300},\,
f_{(001)}=2\alpha_{210},\\
f_{(002)}&=2\alpha_{201},\,
f_{(011)}=2\alpha_{120},\\
f_{(012)}&=\alpha_{111},\,\hskip 2mm
f_{(022)}=2\alpha_{102},\\
f_{(111)}&=6\alpha_{030},\,
f_{(112)}=2\alpha_{021},\\
f_{(122)}&=2\alpha_{012},\,
f_{(222)}=6\alpha_{003}.
\end{split}\;\right\}
\end{equation}

\noindent It follows from \eqref{point}, \eqref{dF}, \eqref{dou}, \eqref{thr} that

\begin{equation}\label{one}
f_{(i)}(c)=\begin{cases}
-1&\hskip -2.5mm\mbox{при $(i)=(0)$},\\
3&\hskip -2.5mm\mbox{при $(i)=(1), (2)$.}
\end{cases}
\end{equation}

\begin{equation}\label{two}
f_{(ij)}(c)=\begin{cases}
0&\hskip -2.5mm\mbox{при $(ij)=(00), (12)$},\\
1&\hskip -2.5mm\mbox{при $(ij)=(01)$},\\
-1&\hskip -2.5mm\mbox{при $(ij)=(02)$},\\
6&\hskip -2.5mm\mbox{при $(ij)=(22)$},\\
-6&\hskip -2.5mm\mbox{при $(ij)=(11)$},
\end{cases}
\end{equation}

\begin{equation}\label{three}
f_{(ijl)}(c)=\begin{cases}
6&\hskip -2.5mm\mbox{при $(ijl)=(000), (111), (222)$},\\
1&\hskip -2.5mm\mbox{при $(ijl)=(012)$},\\
0&\hskip -2.5mm\mbox{при $(ijl)=(001), (002), (011), (022), (112), (122)$}.\\
\end{cases}
\end{equation}

In view of \eqref{eh} the polynomial $h$ has the form
\begin{equation*}
h=f_{(00)}f_{(11)}f_{(22)} +2f_{(01)}f_{(12)}f_{(02)}
-f_{(02)}^2f_{(11)}-f_{(12)}^2f_{(00)}-f_{(01)}^2f_{(22)},
\end{equation*}
from which it follows that for any $i$ the following equality holds
\begin{equation}\label{hi}
\begin{aligned}
h_{(i)}&=f_{(i00)}f_{(11)}f_{(22)}+f_{(00)}f_{(i11)}f_{(22)}+f_{(00)}f_{(11)}f_{(i22)}\\
&\hskip 4.4mm +2f_{(i01)}f_{(12)}f_{(02)}+2f_{(01)}f_{(i12)}f_{(02)}+2f_{(01)}f_{(12)}f_{(i02)}\\
&\hskip 4.4mm -2f_{(02)}f_{(i02)}f_{(11)}-f_{(02)}^2f_{(i11)}\\
&\hskip 4.4mm -2f_{(12)}f_{(i12)}f_{(00)}-f_{(12)}^2f_{(i00)}\\
&\hskip 4.4mm -2f_{(01)}f_{(i01)}f_{(22)}-f_{(01)}^2f_{(i22)}.
\end{aligned}
\end{equation}
From \eqref{hi}, \eqref{two}, \eqref{three} we get
\begin{equation}\label{hhh}
h_{(0)}(c)=-218,\;h_{(1)}(c)=-18.
\end{equation}
From \eqref{one}, \eqref{hhh}, and \eqref{ac} it now follows that \eqref{li} holds as claimed. This completes the proof. \hfill $\Box$

\end{document}